\documentclass[12pt,reqno]{amsart}
\usepackage{amsmath}
\usepackage{amsthm}
\usepackage{amsfonts}
\usepackage{amssymb}
\usepackage{amsxtra}
\usepackage{mathrsfs}
\usepackage{paralist}
\usepackage[hypertex]{hyperref}
\newtheorem{theorem}{Theorem}
\newtheorem{lemma}[theorem]{Lemma}

\newtheorem{corollary}[theorem]{Corollary}
\theoremstyle{definition}

\newtheorem{example}{Example}
\newtheorem{remark}{Remark}
\makeatletter
\newenvironment{proof*}[1][\proofname]{\par
  \pushQED{\qed}%
  \normalfont \topsep6\p@\@plus6\p@\relax
  \trivlist
  \item[\hskip\labelsep
        \itshape
    #1\@addpunct{.}]\ignorespaces
}{%
  \endtrivlist\@endpefalse
}
\makeatother

\DeclareMathOperator{\Ext}{Ext}
\DeclareMathOperator{\Tor}{Tor}
\renewcommand*{\Im}{\mathop{\mathrm{Im}}}

\newcommand*{\Ptens}{\mathop{\widehat\otimes}}

\newcommand*{\ptens}[1]{\mathop{\widehat\otimes}_{#1}}
\newcommand*{\tens}[1]{\mathop{\otimes}_{#1}}
\newcommand*{\lar}{\leftarrow}

\newcommand*{\xla}{\xleftarrow}

\newcommand{\lriso}{\stackrel{\textstyle\sim}{\smash\longrightarrow
\vphantom{\scriptscriptstyle{_1}}}}

\newcommand*{\db}{\mathop{\mathrm{db}}}
\newcommand*{\wdb}{\mathop{\mathrm{w.db}}}
\newcommand*{\dg}{\mathop{\mathrm{dg}}}
\newcommand*{\wdg}{\mathop{\mathrm{w.dg}}}
\renewcommand*{\dh}{\mathop{\mathrm{dh}}}
\newcommand*{\wdh}{\mathop{\mathrm{w.dh}}}
\newcommand*{\lmod}{\mbox{-}\!\mathop{\mathsf{mod}}}

\newcommand*{\rmod}{\mathop{\mathsf{mod}}\!\mbox{-}}

\newcommand*{\bimod}{\mbox{-}\!\mathop{\mathsf{mod}}\!\mbox{-}}

\newcommand*{\id}{\mathbf{1}}
\newcommand*{\CC}{\mathbb C}
\newcommand*{\DD}{\mathbb D}
\newcommand*{\R}{\mathbb R}

\newcommand*{\N}{\mathbb N}
\newcommand*{\Z}{\mathbb Z}
\newcommand*{\h}{\mathbf h}
\newcommand*{\cH}{\mathscr H}

\newcommand*{\eps}{\varepsilon}

\newcommand*{\Fr}{\mathsf{Fr}}

\newcommand*{\bfU}{\textbf{U}}
\newcommand*{\bfN}{\textbf{N}}
\newcommand*{\bfB}{\textbf{B}}
\newcommand*{\bfM}{\textbf{M}}
{\end{compactenum}}

\begin{document}
\title{Homological dimensions of K\"othe algebras}
\author{A. Yu. Pirkovskii}
\address{Department of Nonlinear Analysis and Optimization\\
Faculty of Science\\
Peoples' Friendship University of Russia\\
Mikluho-Maklaya 6\\
117198 Moscow\\
Russia}
\email{pirkosha@sci.pfu.edu.ru, pirkosha@online.ru}
\thanks{Partially supported by the RFBR grant 08-01-00867.}
\date{}

\begin{abstract}
Given a metrizable K\"othe algebra $\lambda(P)$, we compute the global dimension,
the weak global dimension, the bidimension, and the weak bidimension of $\lambda(P)$
in terms of the K\"othe set $P$.
\end{abstract}

\maketitle

K\"othe sequence spaces play a significant r\^ole in modern functional analysis.
On the one hand, the class of K\"othe spaces is rather large and contains many
important spaces of smooth functions and distributions \cite{Vogt_seq_rep}.
On the other hand, K\"othe spaces are often used to provide various examples
and counterexamples in the theory of topological vector spaces.

As was observed in \cite{Bh_Deh}, many K\"othe spaces can be viewed as topological
algebras under pointwise multiplication. The study of homological properties of K\"othe
algebras was initiated by the author in \cite{Pir_bipr,Pir_bipr2,Pir_msb}.
In particular, necessary and sufficient conditions for a K\"othe algebra to be biprojective
(or, equivalently, biflat) were obtained, and homological dimensions of biprojective
K\"othe algebras were computed in some special cases.

In this note we complete the study of homological dimensions of metrizable
K\"othe algebras by computing the global dimension, the weak global dimension, the bidimension,
and the weak bidimension of a K\"othe algebra $\lambda(P)$ in terms
of the K\"othe set $P$. In fact, a great deal of this work was done in
\cite{Pir_bipr,Pir_bipr2,Pir_msb}. However, the following two questions were left unanswered:
(1) Does there exist a biprojective K\"othe algebra $\lambda(P)$ whose bidimension
equals $2$, while the global dimension equals $1$?; and (2) What can be said about homological
dimensions of nonbiprojective K\"othe algebras? It is the aim of this paper to
answer the above questions.

\section{Preliminaries}

Throughout, all vector spaces and algebras are assumed to be over the field $\CC$
of complex numbers. All algebras are assumed to be associative, but
not necessarily unital. The unitization of an algebra $A$ is denoted by $A_+$.

By a {\em $\Ptens$-algebra} we mean an algebra $A$ endowed with
a complete locally convex topology in such a way that the product map
$A\times A\to A$ is jointly continuous.
Note that the above map uniquely extends to a continuous linear map
${A\Ptens A\to A},\; a\otimes b\mapsto ab$,
where the symbol $\Ptens$ stands for the completed projective
tensor product (whence the name ``$\Ptens$-algebra'').
If the topology on $A$ can be determined by a family of submultiplicative
seminorms (i.e., a family $\{\|\cdot\|_\lambda : \lambda\in\Lambda\}$ of
seminorms such that $\| ab\|_\lambda\le\| a\|_\lambda \| b\|_\lambda$
for all $a,b\in A$), then $A$ is said to be {\em locally $m$-convex}
(or an {\em Arens-Michael algebra}).
A {\em Fr\'echet algebra} is a $\Ptens$-algebra $A$ whose underlying locally convex
space is a Fr\'echet space (unlike some authors, we do not assume $A$ to be
locally $m$-convex).

Let $I$ be any set, and let $P$ be a set of nonnegative real-valued functions on $I$.
For $p\in P$, we shall write $p_i$ for $p(i)$.
Recall that $P$ is a {\em K\"othe set} on $I$ if the following axioms are
satisfied:
\begin{align*}
\tag*{(P1)}
&\forall\,i\in I\quad\exists\, p\in P:\quad p_i>0\, ;\\
\tag*{(P2)}
&\forall\, p,q\in I\quad\exists\, r\in P:\quad\max\{ p_i,q_i\}\le r_i\quad\forall\, i\in I\, .
\end{align*}
Given a K\"othe set $P$, the {\em K\"othe space}
$\lambda(P)$ is defined as follows:
\begin{equation*}
\lambda(P)=
\Bigl\{ x=(x_i)\in \CC^I :
\| x\|_p=\sum_i |x_i|p_i <\infty\quad\forall\, p\in P\Bigr\}\, .
\end{equation*}
This is a complete locally convex space
with the topology determined by
the family of seminorms $\{\|\cdot\|_p : p\in P\}$. Clearly, $\lambda(P)$
is a Fr\'echet space if and only if $P$ contains an at most countable
cofinal subset.

For each $i\in I$ denote by $e_i$ the function on $I$ which is $1$ at $i$,
$0$ elsewhere. Obviously, $x=\sum_i x_i e_i$ for each $x\in\lambda(P)$.

Given a K\"othe set $P$ on $I$, each $n$-tuple
$(p^1,\ldots ,p^n)\in P^n$ determines a function on $I^n$ by
$(i_1,\ldots ,i_n)\mapsto p^1_{i_1}\cdots p^n_{i_n}$.
The set of all such functions will be denoted by $P^{\times n}$.
Clearly, $P^{\times n}$ is a K\"othe set on $I^n$.
By \cite{Pietsch2}, there exists a topological isomorphism
\begin{equation}
\label{Kothe_ptens}
\begin{split}
\lambda(P)^{\Ptens n}=
\underbrace{\lambda(P)\Ptens\cdots\Ptens\lambda(P)}_n
&\lriso \lambda(P^{\times n});\\
\qquad e_{i_1}\otimes\cdots\otimes e_{i_n} &\mapsto
e_{i_1\ldots i_n}.
\end{split}
\end{equation}
It is easy to see that the topology on $\lambda(P^{\times n})$ is determined
by the family of seminorms $\{\|\cdot\|_{(p,\ldots ,p)} : p\in P\}$.
We will denote the above seminorm simply by $\|\cdot\|_p$; this should not
cause any confusion.

If $P,Q$ are K\"othe sets, we write $P\prec Q$ if for each
$p\in P$ there exist $q\in Q$ and $C>0$ such that $p_i\le Cq_i$
for all $i\in I$. This is equivalent to say that $\lambda(Q)\subset\lambda(P)$,
and the embedding of $\lambda(Q)$ into $\lambda(P)$ is continuous.
If $P\prec Q$ and $Q\prec P$ (i.e., if $\lambda(P)=\lambda(Q)$ topologically),
we write $P\sim Q$.
We set
\begin{align*}
P\cdot Q&=\{ pq=(p_i q_i)_{i\in I} : p\in P,\; q\in Q\},\\
P^{[\alpha]}&=\{ p^\alpha=(p_i^\alpha)_{i\in I} : p\in P\}\quad (\alpha>0).
\end{align*}
Note that $P^{[2]}\sim P\cdot P$ by $(P2)$.

It is easy to see that $P\prec P^{[2]}$ if and only if for each
$a,b\in\lambda(P)$ the pointwise product $ab$ belongs to $\lambda(P)$,
and for each $p\in P$ there exist $q\in P$ and $C>0$ such that
$\| ab\|_p\le C\| a\|_q \| b\|_q$ for every $a,b\in\lambda(P)$.
Hence the above condition implies that
$\lambda(P)$ is a $\Ptens$-algebra under pointwise multiplication.
Note that this condition is satisfied automatically whenever $p_i\ge 1$
for each $p\in P$ and each $i\in I$; moreover, in this case
$\lambda(P)$ is locally $m$-convex.
Algebras of the form $\lambda(P)$ (where $P$ is any K\"othe set
satisfying $P\prec P^{[2]}$) are called {\em K\"othe algebras}.

\begin{example}
\label{example:l^1}
The Banach algebra $\ell^1(I)$ is clearly a K\"othe algebra.
\end{example}

\begin{example}
\label{example:C^I}
The algebra $\CC^I$ endowed with the direct product topology
is a K\"othe algebra.
To see this, it suffices to set $P$ to be the family of all nonnegative
functions with finite support. It is also clear that $\CC^I$ is locally
$m$-convex.
\end{example}

\begin{example}
\label{example:Lambda}
Fix a real number $0<R\le\infty$ and a nondecreasing sequence
$\alpha=(\alpha_n)_{n\in\N}$ of positive numbers with $\lim_n\alpha_n=\infty$.
The {\em power series space} $\Lambda_R(\alpha)$ is the set of all
complex sequences $x=(x_n)_{n\in\N}$ such that
$\| x\|_r=\sum_n |x_n|r^{\alpha_n} <\infty$ for all $0<r<R$.
Evidently, $\Lambda_R(\alpha)$ is a metrizable K\"othe space.
If $R\ge 1$, then $\Lambda_R(\alpha)$ satisfies condition $P\prec P^{[2]}$ and is
therefore a K\"othe algebra. Moreover, since the seminorms
$\|\cdot\|_r$ are submultiplicative for $r\ge 1$, we see that $\Lambda_R(\alpha)$
is locally $m$-convex provided that $R>1$.
\end{example}

Here are two special cases of Example \ref{example:Lambda}.

\begin{example}
\label{example:s}
If $\alpha_n=\log n$, then $\Lambda_\infty(\alpha)$ is topologically isomorphic
to the space of {\em rapidly decreasing sequences}
\[
s=\Bigl\{ x=(x_n)\in\CC^\N :
\| x\|_k=\sum_n |x_n| n^k <\infty\quad\forall\, k\in \N\Bigr\}\, .
\]
\end{example}

\begin{example}
\label{example:H(D)}
If $\alpha_n=n$, then $\Lambda_R(\alpha)$ is topologically isomorphic to the space
of functions holomorphic on the disc
$\mathbb D_R=\{ z\in\CC : |z| < R\}$ of radius $R$.
Under this identification, the multiplication in $\Lambda_R(\alpha)$
corresponds to the ``componentwise'' product of the Taylor expansions
of holomorphic functions (the {\em Hadamard product}; see \cite{Render}).
The resulting topological algebra is denoted by $\cH(\mathbb D_R)$.
\end{example}

We now recall some basic facts from the homology theory of $\Ptens$-algebras.
For details, see \cite{X1,X_HOA,T1}. Some details on weak homological dimensions
can also be found in \cite{Sel_bifl,Pir_msb}.

Let $A$ be a $\Ptens$-algebra. By a {\em left $A$-$\Ptens$-module} we mean a
left $A$-module endowed with a complete locally convex topology in such a way that
the action $A\times X\to X$ is jointly continuous. If $X$ and $Y$ are left
$A$-$\Ptens$-modules, then the space of all continuous $A$-module morphisms
from $X$ to $Y$ is denoted by ${_A}\h(X,Y)$. Right $A$-$\Ptens$-modules
and $A$-$\Ptens$-bimodules are defined similarly.
The category of left $A$-$\Ptens$-modules (respectively, right $A$-$\Ptens$-modules,
$A$-$\Ptens$-bimodules) and continuous $A$-module morphisms
will be denoted by $A\lmod$ (respectively, $\rmod A$, $A\bimod A$).

If $X$ is a right $A$-$\Ptens$-module and $Y$
is a left $A$-$\Ptens$-module, then their $A$-module tensor product
$X\ptens{A}Y$ is defined to be
the completion of the quotient $(X\Ptens Y)/N$, where $N\subset X\Ptens Y$
is the closed linear span of all elements of the form
$x\cdot a\otimes y-x\otimes a\cdot y$
($x\in X$, $y\in Y$, $a\in A$).
As in pure algebra, the $A$-module tensor product can be characterized
by a universal property (see \cite{X1} for details).

A chain complex $C_\bullet=(C_n,d_n)_{n\in\Z}$ in $A\lmod$ is {\em admissible} if
it splits in the category of topological vector spaces.
A left $A$-$\Ptens$-module $P$ is {\em projective} if the functor
${_A}\h(P,-)$ is {\em exact} in the sense that for every admissible chain
complex $C_\bullet$ in $A\lmod$ the complex ${_A}\h(P,C_\bullet)$ of vector
spaces is exact. Projective right $A$-$\Ptens$-modules and projective
$A$-$\Ptens$-bimodules are defined similarly.
A $\Ptens$-algebra $A$ is {\em biprojective} if $A$ is projective
in $A\bimod A$. A {\em resolution} of $X\in A\lmod$ is a chain complex
$P_\bullet=(P_n,d_n)_{n\ge 0}$ together with a morphism $\eps\colon P_0\to X$
such that $0\lar X\xla{\eps} P_\bullet$ is an admissible complex.
If $P_n$ is projective for each $n\ge 0$, then $(P_\bullet,\eps)$ is a
{\em projective resolution}. It is known that the category $A\lmod$ has
{\em enough projectives}, i.e., every $X\in A\lmod$ has a projective resolution.
The {\em homological dimension} of $X\in A\lmod$ is the minimum integer $n=\dh_A X$
with the property that $X$ has a projective resolution $(P_\bullet,\eps)$
with $P_i=0$ for all $i>n$. If no such $n$ exists, we set $\dh_A X=\infty$.
The {\em global dimension} of $A$ is defined by
\[
\dg A=\sup\{ \dh\nolimits_A X : X\in A\lmod\}.
\]
The {\em bidimension} of $A$ is defined to be the homological dimension of $A_+$
in $A\bimod A$. We always have $\dg A\le\db A$. Algebras $A$ with $\db A=0$ are
called {\em contractible}. Equivalently, $A$ is contractible if and only if
$A$ is biprojective and unital.

Now let $A$ be a Fr\'echet algebra, and let $A\lmod(\Fr)$ denote the full subcategory
of $A\lmod$ consisting of left Fr\'echet $A$-modules. The categories of right
Fr\'echet $A$-modules and of Fr\'echet $A$-bimodules will be denoted by
$\rmod A(\Fr)$ and $A\bimod A(\Fr)$, respectively.
By using \cite[Theorem III.1.27]{X1}, it is easy to see that a left Fr\'echet
$A$-module $P$ is projective if and only if $P$ is projective in $A\lmod(\Fr)$
(in the sense that the functor ${_A}\h(P,-)$ is exact on $A\lmod(\Fr)$).
Together with \cite[Theorem III.5.4]{X1}, this implies that for each $X\in A\lmod(\Fr)$
the homological dimension $\dh_A X$ does not depend on whether we compute it in $A\lmod$
or in $A\lmod(\Fr)$. The same is true of $\db A$ (as for $\dg A$, we do not know
the answer; see also Remark~\ref{rem:dg_Fr_Ptens} below).

A left Fr\'echet $A$-module $F$ is {\em flat} if the functor
$(-)\ptens{A} F$ is exact on $\rmod A(\Fr)$, i.e., if for every admissible chain
complex $C_\bullet$ in $\rmod A(\Fr)$ the complex $C_\bullet\ptens{A} F$ of vector
spaces is exact. A Fr\'echet algebra $A$ is {\em biflat} if $A$ is flat
in $A\bimod A(\Fr)$. A resolution $(P_\bullet,\eps)$ of $X$ in $A\lmod(\Fr)$ is a {\em flat
resolution} if $P_n$ is flat for each $n\ge 0$.
The {\em weak homological dimension} of $X\in A\lmod(\Fr)$
is the minimum integer $n=\wdh_A X$
with the property that $X$ has a flat resolution $(P_\bullet,\eps)$
with $P_i=0$ for all $i>n$. If no such $n$ exists, we set $\wdh_A X=\infty$.
The {\em weak global dimension} of $A$ is defined by
\[
\wdg A=\sup\{ \wdh\nolimits_A X : X\in A\lmod(\Fr)\}.
\]
The {\em weak bidimension} of $A$ is defined to be the weak homological dimension of $A_+$
in $A\bimod A(\Fr)$. We always have $\wdg A\le\wdb A$. Algebras $A$ with $\wdb A=0$ are
called {\em amenable}.

Since each projective Fr\'echet $A$-module is flat, we have
$\wdh_A X\le\dh_A X$ for every $X\in A\lmod(\Fr)$. Consequently,
$\wdg A\le\dg A$ and $\wdb A\le\db A$ for each Fr\'echet algebra $A$.

Suppose that $X$ is a right Fr\'echet $A$-module and $Y$ is a left
Fr\'echet $A$-module. The space $\Tor_n^A(X,Y)$ is defined to be the
$n$th homology of the complex $X\ptens{A} Q_\bullet$, where $Q_\bullet$
is a flat resolution of $Y$ in $A\lmod(\Fr)$. Equivalently,
$\Tor_n^A(X,Y)$ is the
$n$th homology of the complex $P_\bullet\ptens{A} Y$, where $P_\bullet$
is a flat resolution of $X$ in $\rmod A(\Fr)$.
The above definitions do not depend on the particular choice of $P_\bullet$
and $Q_\bullet$. We have $\wdh_A Y\le n$ if and only if for each $X\in \rmod A(\Fr)$
$\Tor_{n+1}^A(X,Y)=0$ and $\Tor_n^A(X,Y)$ is Hausdorff.
There is a similar characterization of $\dh_A Y$ in terms of $\Ext$ spaces,
but we will not use it in the sequel.

\begin{remark}
When dealing with flat modules and weak dimensions, we deliberately restrict
ourselves to Fr\'echet modules and Fr\'echet algebras. For an explanation,
see \cite[Remark 2.5]{Pir_msb} and \cite[Remark 7.2]{Pir_cycflat2}.
\end{remark}

\section{Biprojective K\"othe algebras}

Let $\lambda(P)$ be a K\"othe algebra. Throughout we will use the following
conditions (\bfU), (\bfN), (\bfB), and (\bfM) on the K\"othe set $P$
(``\bfU'' is for ``unital'', ``\bfN'' is for ``nuclear'', ``\bfB''
is for ``biprojective''
or ``biflat'', and ``\bfM'' is for ``matrices''):

\pltopsep=8pt
\plitemsep=4pt
\begin{compactenum}
\item[(\bfU)] $\forall p\in P\; \sum_i p_i<\infty$.
\item[(\bfN)] $\forall p\in P\; \exists\, q\in P\;\exists\,\alpha\in\ell^1\; :\; p\le\alpha q$.
\item[(\bfB)] $P\sim P^{[2]}$.
\item[(\bfM)] There exist complex matrices $\alpha=(\alpha_{ij})_{i,j\in I}$
and $\beta=(\beta_{ij})_{i,j\in I}$ such that
\plitemsep=3pt
\begin{compactenum}
\item[(\bfM1)]
$\alpha_{ij}+\beta_{ij}=1\quad (i,j\in I)$;
\item[(\bfM2)]
$\forall p\in P\quad\exists\, C>0\quad\exists\, q\in P\quad\forall j\in\N\quad
\sup_i |\alpha_{ij}|p_i p_j\le Cq_j$;
\item[(\bfM3)]
$\forall p\in P\quad\exists\, C>0\quad\exists\, q\in P\quad\forall i\in\N\quad
\sup_j |\beta_{ij}|p_j p_i\le Cq_i$.
\end{compactenum}
\end{compactenum}

Clearly, condition (\bfU) means exactly that $\lambda(P)$ is unital.
By the Grothendieck-Pietsch criterion, (\bfN) is equivalent to
$\lambda(P)$ being nuclear. Condition (\bfB) holds if and only if $\lambda(P)$
is biprojective \cite[Theorem 3.5]{Pir_bipr}. Moreover, if $\lambda(P)$ is metrizable, then
(\bfB) is equivalent to $\lambda(P)$ being biflat \cite[Theorem 5.2]{Pir_msb}.
Finally, if (\bfB) holds, then (\bfN)\&(\bfM) $\iff\db\lambda(P)\le 1$
\cite[Corollary 4.4 and Theorem 4.7]{Pir_bipr}.
By \cite[Proposition 6.11]{Pir_msb}, (\bfU) is also equivalent to
$\lambda(P)$ being contractible.
Therefore (\bfU) implies (\bfB), (\bfN), and (\bfM).

Let $P$ be a K\"othe set on $I$. For each $p\in P$ we define a function
$\bar p\colon I\to\R_+$
by $\bar p_i=\min\{ p_i,1\}$. Clearly, $\bar P=\{ \bar p : p\in P\}$ is a K\"othe set.

\begin{lemma}
\label{lemma:Pbar}
If $P\prec P^{[2]}$, then $\bar P\prec P\cdot\bar P$.
\end{lemma}
\begin{proof}
Given $p\in P$, choose $q\in P$ and $C\ge 1$ such that $p\le Cq^2$ and $p\le q$.
Fix any $i\in I$. If $q_i<1$, then $p_i<1$, so that $\bar p_i=p_i$, $\bar q_i=q_i$, whence
$\bar p_i\le Cq_i\bar q_i$. If $q_i\ge 1$, then $\bar q_i=1$, and
$\bar p_i\le 1\le Cq_i=Cq_i\bar q_i$. Therefore $\bar p\le Cq\bar q$, which proves
the claim.
\end{proof}

\begin{corollary}
Let $\lambda(P)$ be a K\"othe algebra. Then for each $a\in\lambda(P)$
and each $x\in\lambda(\bar P)$ the pointwise product $a\cdot x$ is in $\lambda(\bar P)$.
Moreover, for each $\bar p\in \bar P$ there exist $q\in P$ and $C>0$ such that
$\| a\cdot x\|_{\bar p}\le C\| a\|_q \| x\|_{\bar q}$ for every $a\in\lambda(P)$,
$x\in\lambda(\bar P)$. Therefore $\lambda(\bar P)$
is a $\lambda(P)$-$\Ptens$-module under pointwise multiplication.
\end{corollary}

We now state the main result of this section.

\begin{theorem}
\label{thm:dhdb_bipr}
Let $A=\lambda(P)$ be a K\"othe algebra satisfying {\upshape (\bfB)}
and {\upshape (\bfN)}. Suppose that $\dh_A\lambda(\bar P)\le 1$. Then $\db A\le 1$.
\end{theorem}

To prove Theorem \ref{thm:dhdb_bipr}, we need some preparation.
Let $A$ be a $\Ptens$-algebra, and let $X$ be a left $A$-$\Ptens$-module.
Following \cite{X1}, we set $X_\Pi=A\ptens{A} X$ and define
\begin{equation}
\label{kappa}
\varkappa_X\colon X_\Pi\to X,\quad
a\otimes x\mapsto a\cdot x\quad (a\in A,\; x\in X).
\end{equation}
Suppose now that $A$ is biprojective.
Then, by a result of Helemskii \cite[Section V.2]{X1},
$\dh_A X\le 1$ if and only if the ``diagonal'' map
\begin{equation}
\label{diag}
A\Ptens X_\Pi \to (A_+\Ptens X_\Pi)\oplus (A\Ptens X),\quad
a\otimes x \mapsto (a\otimes x,a\otimes\varkappa_X(x))
\end{equation}
is a coretraction in $A\lmod$.

To apply the above result to $A=\lambda(P)$ and $X=\lambda(\bar P)$, we first
have to describe $X_\Pi$ explicitly.

\begin{lemma}
\label{lemma:Pbar2}
Let $A=\lambda(P)$ be a K\"othe algebra satisfying {\upshape (\bfB)},
and let $X=\lambda(\bar P)$.
Then $X_\Pi$ is isomorphic to $A$ in $A\lmod$. Under this identification,
the canonical map $\varkappa_X\colon X_\Pi\to X$ becomes the identity embedding
of $\lambda(P)$ into $\lambda(\bar P)$.
\end{lemma}
\begin{proof}
By \cite[Lemma 6.4]{Pir_msb},
\[
X_\Pi\cong\Bigl\{ x=(x_i)\in\prod_i (e_i X) :
\| x\|_{p,q}=\sum_i \|x_i\|_{\bar p}\, q_i<\infty
\;\forall p,q\in P\Bigr\},
\]
and, under this identification, the canonical map $\varkappa_X$ takes each
$x=(x_i)\in X_\Pi$ to $\sum_i x_i\in X$. Since $e_i X=\CC e_i$, it follows that
\[
X_\Pi\cong\Bigl\{ a=(a_i)\in\CC^I :
\| a\|_{p,q}=\sum_i |a_i|\bar p_i q_i<\infty
\;\forall p,q\in P\Bigr\}=\lambda(\bar P\cdot P).
\]
Therefore we need only prove that $\bar P\cdot P\sim P$.
Since $P^{[2]}\prec P$ and $\bar P\prec P$, we have $\bar P\cdot P\prec P$.
For the converse, take any $p\in P$ and choose $C\ge 1$ and $q\in P$
such that $p\le Cq^2$ and $p\le q$. Fix any $i\in I$. If $q_i<1$, then $\bar q_i=q_i$,
and so $p_i\le C\bar q_i q_i$. If $q_i\ge 1$, then $\bar q_i=1$, and so
$p_i\le q_i\le Cq_i=C\bar q_i q_i$. Therefore $p\le C\bar q q$, so that
$P\prec\bar P\cdot P$, and, finally, $\bar P\cdot P\sim P$, as required.
\end{proof}

\begin{proof}[Proof of Theorem \ref{thm:dhdb_bipr}]
Set $X=\lambda(\bar P)$. Identifying $X_\Pi$ with $A$ by Lemma~\ref{lemma:Pbar2},
we see that the canonical map \eqref{diag} becomes
\begin{equation}
\label{diag2}
A\Ptens A \to (A_+\Ptens A)\oplus (A\Ptens X),\quad
a\otimes b \mapsto (a\otimes b,a\otimes b).
\end{equation}
Since $A$ is biprojective (see \cite[Theorem 3.5]{Pir_bipr})
and $\dh_A X\le 1$, it follows that
\eqref{diag2} is a coretraction in $A\lmod$.
Therefore there exists a continuous linear map
$\varphi\colon A\to A\Ptens A$ and an $A$-module morphism
$\psi\colon A\Ptens X\to A\Ptens A$ such that
\begin{equation}
\label{phi-psi}
a\otimes b=a\cdot\varphi(b)+\psi(a\otimes b)\quad (a,b\in A).
\end{equation}
Since $\varphi$ and $\psi$ are continuous, for each $p\in P$ there exist
$q\in Q$ and $C>0$ such that
\begin{align}
\label{phi_cont}
\|\varphi(a)\|_{p,p} &\le C\| a\|_q\quad (a\in A);\\
\label{psi_cont}
\|\psi(a\otimes x)\|_{p,p} &\le C\| a\|_q \| x\|_{\bar q} \quad (a\in A,\; x\in X).
\end{align}
Using the isomorphism $A\Ptens A\cong\lambda(P^{\times 2})$ (see \eqref{Kothe_ptens}),
we may represent each $\varphi(e_j)$ as
\begin{equation}
\label{phi_coeff}
\varphi(e_j)=\sum_{k,l}\lambda_{klj} e_k\otimes e_l\quad (j\in I).
\end{equation}
Then \eqref{phi_cont} implies that
\begin{equation}
\label{phi_cont_coeff}
\sum_{k,l} |\lambda_{klj}| p_k p_l \le Cq_j \quad (j\in I).
\end{equation}
Now fix any $i\in I$ and define
\[
\psi_i\colon X\to A\Ptens A,\quad \psi_i(x)=\psi(e_i\otimes x) \quad (x\in X).
\]
Since $\psi$ is an $A$-module morphism, we have
\[
e_j\cdot \psi_i(x)=\delta_{ji} \psi_i(x) \quad (i,j\in I),
\]
and so $\Im\psi_i\subset e_i\cdot A\Ptens A=\CC e_i\otimes A$.
Therefore for each $i\in I$ there exists a linear map $f_i\colon X\to A$
such that
\begin{equation}
\label{f_i_def}
\psi(e_i\otimes x)=\psi_i(x)=e_i\otimes f_i(x)\quad (x\in X).
\end{equation}
Setting $a=e_i$ and $b=e_j$ in \eqref{psi_cont}, we see that
\begin{equation}
\label{psi_cont_coeff_1}
\| f_i(e_j)\|_p\, p_i \le Cq_i \bar q_i \le Cq_i \quad (i,j\in I).
\end{equation}
Let
\begin{equation}
\label{f_coeff}
f_i(e_j)=\sum_{k} \mu_{kij} e_k \quad (i,j\in I).
\end{equation}
Then \eqref{psi_cont_coeff_1} is equivalent to
\begin{equation}
\label{psi_cont_coeff_2}
\sum_k |\mu_{kij}| p_k p_i \le Cq_i \quad (i,j\in I).
\end{equation}
Setting $a=e_i$ and $b=e_j$ in \eqref{phi-psi} and taking into account
\eqref{phi_coeff}, \eqref{f_i_def}, and \eqref{f_coeff}, we see that
\[
e_i\otimes e_j=e_i\cdot\sum_{k,l}\lambda_{klj} e_k\otimes e_l+e_i\otimes\sum_k \mu_{kij} e_k,
\]
which is equivalent to
\begin{equation}
\label{lambda_mu}
e_j=\sum_l \lambda_{ilj} e_l + \sum_k \mu_{kij} e_k = \sum_k (\lambda_{ikj}+\mu_{kij}) e_k
\quad (i,j\in I).
\end{equation}
Now set
\[
\alpha_{ij}=\lambda_{ijj},\quad \beta_{ij}=\mu_{jij} \quad (i,j\in I).
\]
Then \eqref{lambda_mu} implies that
\[
\alpha_{ij}+\beta_{ij}=1 \quad (i,j\in I),
\]
i.e., (\bfM1) holds. Next, \eqref{phi_cont_coeff} implies that
\[
|\alpha_{ij}| p_i p_j \le Cq_j \quad (i,j\in I),
\]
i.e., (\bfM2) holds. Finally, \eqref{psi_cont_coeff_2} implies that
\[
|\beta_{ij}| p_j p_i \le Cq_i \quad (i,j\in I),
\]
i.e., (\bfM3) holds. Thus $P$ satisfies (\bfM), and so $\db\lambda(P)\le 1$
by \cite[Theorem 4.7]{Pir_bipr}.
\end{proof}

Combining Theorem \ref{thm:dhdb_bipr} with our earlier results obtained in \cite{Pir_bipr}
and \cite{Pir_msb} yields a complete classification of biprojective K\"othe algebras
by their homological dimensions $\dg$ and $\db$. Before formulating the result,
let us recall some notation. Let $P$ be a K\"othe set on $I$. The K\"othe space
$\lambda^\infty(P)$ is defined by
\begin{equation*}
\lambda^\infty(P)=
\Bigl\{ a=(a_i)\in \CC^I :
\| a\|_p^\infty=\sup_i |a_i|p_i <\infty\quad\forall\, p\in P\Bigr\}\, .
\end{equation*}
This is a complete locally convex space
with the topology determined by
the family of seminorms $\{\|\cdot\|_p^\infty : p\in P\}$.
Clearly, $\lambda(P)\subset\lambda^\infty(P)$, and the embedding is continuous.
By the Grothendieck-Pietsch criterion, $\lambda(P)$ is nuclear if and only if
$\lambda(P)=\lambda^\infty(P)$ topologically, which is equivalent to condition (\bfN).
If $\lambda(P)$ is a $\Ptens$-algebra under pointwise multiplication
(i.e., if $P\prec P^{[2]}$), then so is $\lambda^\infty(P)$, and the algebra embedding
$\lambda(P)\subset\lambda^\infty(P)$ makes $\lambda^\infty(P)$ into a
$\lambda(P)$-$\Ptens$-module.

Given a $\Ptens$-algebra $A$, we consider $\CC$ as an $A$-$\Ptens$-module by letting
$A$ act on $\CC$ trivially. In other words, $\CC=A_+/A$.

\begin{theorem}
\label{thm:dg_bipr}
Let $A=\lambda(P)$ be a K\"othe algebra satisfying {\upshape (\bfB)}. Then
\[
\dg A=\db A=
\begin{cases}
0, & \text{$P$ satisfies {\upshape (\bfU)}}.\\
1, & \parbox[t]{95mm}{$P$ satisfies {\upshape (\bfN)} and {\upshape (\bfM)},
but does not satisfy {\upshape (\bfU)}. In this case, $\dh_A\CC=1$.}\\
2, & \parbox[t]{95mm}{$P$ satisfies {\upshape (\bfN)}, but does not satisfy
{\upshape (\bfM)}. In this case, $\dh_A\lambda(\bar P)=2$.}\\
2, & \parbox[t]{95mm}{$P$ does not satisfy {\upshape (\bfN)}. In this case,
$\dh_A\lambda^\infty(P)=2$.}
\end{cases}
\]
\end{theorem}
\begin{proof}
By \cite[Theorem 3.5]{Pir_bipr}, condition (\bfB) is equivalent to $A$ being biprojective.
Hence (\bfB) implies that $\db A\le 2$ \cite[Proposition 2.5.8]{X_HOA},
and so $\dg A\le 2$.
The rest follows from Theorem~\ref{thm:dhdb_bipr} and \cite[Theorems 4.3 and 4.7]{Pir_bipr}
together with \cite[Proposition 6.11]{Pir_msb}.
\end{proof}

\section{Nonbiprojective K\"othe algebras}

In this section, we show that the homological dimensions $\dg$, $\db$, $\wdg$, and $\wdb$
of a nonbiprojective metrizable K\"othe algebra are infinite. First we need a lemma.

\begin{lemma}
\label{lemma:Kothe_pow}
Let $P$ be a K\"othe set. Suppose that $P^{[l]}\prec P^{[k]}$
for some $k,l\in\R$, $0<k<l$. Then $P^{[2]}\prec P$.
\end{lemma}
\begin{proof}
Set $r=l/k$. Then $P^{[r]}\prec P$, and, by induction, $P^{[r^n]}\prec P$ for every $n\in\N$.
Fix $n\in\N$ such that $\alpha=r^n\ge 2$. Then for each $p\in P$ there exist
$C\ge 1$ and $q\in P$ such that $p^\alpha\le Cq$ and $p\le q$.
Now fix any $i\in I$. If $p_i\ge 1$, then $p_i^2\le p_i^\alpha \le Cq_i$.
If $p_i<1$, then $p_i^2<p_i\le Cq_i$. Thus $p^2\le Cq$, which proves the claim.
\end{proof}

\begin{theorem}
\label{thm:Torne0}
Let $A=\lambda(P)$ be a metrizable K\"othe algebra not satisfying {\upshape (\bfB)}.
Then for each odd $n\in\N$ we have $\Tor_n^A(\CC,\CC)\ne 0$. Moreover, the latter
space is not Hausdorff. As a corollary,
$\dg A=\db A=\wdg A=\wdb A=\wdh_A\CC=\infty$.
\end{theorem}
\begin{proof}
By \cite[2.3.3]{X_HOA},
the spaces $\Tor_n^A(\CC,\CC)$ are the homology of the chain complex
\[
0 \lar\CC\xla{0} A \lar A\Ptens A \lar \cdots \lar A^{\Ptens n}
\xla{d} A^{\Ptens (n+1)} \lar \cdots,
\]
the differential being given by
\[
\begin{split}
d(a_0\otimes\cdots\otimes a_n)
&=a_0 a_1 \otimes \cdots \otimes a_n\\
&+\sum_{k=1}^{n-2} (-1)^k a_0\otimes \cdots\otimes a_k a_{k+1}\otimes\cdots\otimes a_n\\
&+(-1)^{n-1} a_0\otimes\cdots\otimes a_{n-1} a_n.
\end{split}
\]
Identifying $A^{\Ptens n}$ with $\lambda(P^{\times n})$ (see \eqref{Kothe_ptens}),
we see that for each $n$
the differential $d\colon\lambda(P^{\times (n+1)})\to\lambda(P^{\times n})$
acts by the formula
\begin{equation}
\label{d_Kothe}
d(e_{i_0\ldots i_n})
=\delta_{i_0 i_1} e_{i_1\ldots i_n}
+\sum_{k=1}^{n-2} (-1)^k \delta_{i_k i_{k+1}} e_{i_0\ldots \hat i_{k} \ldots i_n}
+(-1)^{n-1} \delta_{i_{n-1} i_n} e_{i_0\ldots i_{n-2} i_n},
\end{equation}
where, as usual, the notation $\hat i_k$ indicates that $i_k$ is omitted.
For notational convenience, set $e_i^n=e_{i\ldots i}$ with the subscript ``$i$''
repeated $n$ times. Then \eqref{d_Kothe} implies that
\begin{equation}
\label{d_e_i}
d(e_i^{n+1})=\left\{
\begin{array}{ll}
e_i^n, & n \text{ is odd},\\
0, & n \text{ is even}.
\end{array}
\right.
\end{equation}
For each $n\in\N$ and each $i\in I$ set
\[
E_n^i=\left\{\sum\alpha_{k_1\ldots k_n} e_{k_1\ldots k_n}\in\lambda(P^{\times n}) :
\alpha_{i\ldots i}=0\right\}.
\]
It follows from \eqref{d_Kothe} that $d(e_{k_0\ldots k_n})\in E_n^i$ unless
$(k_0,\ldots ,k_n)=(i,\ldots ,i)$. Therefore
\begin{equation}
\label{E_incl}
d(E^i_{n+1})\subset E^i_n.
\end{equation}

Suppose that $n$ is odd, and assume, towards a contradiction, that $\Tor^A_n(\CC,\CC)$
is Hausdorff. This is equivalent to say that the image of
$d\colon\lambda(P^{\times (n+1)})\to\lambda(P^{\times n})$ is closed.
By the Open Mapping Theorem, $d$ is an open map onto its image.
Therefore for each $p\in P$ there exist $q\in P$ and $C>0$ such that
for each $y\in\Im d$ there exists $x\in d^{-1}(y)$ satisfying $\| x\|_p\le C\| y\|_q$.

Now fix any $i\in I$, set $y=e_i^n$ (which belongs to $\Im d$ by \eqref{d_e_i}),
and find $x\in d^{-1}(y)$ as above. Since
$\lambda(P^{\times (n+1)})=\CC e_i^{n+1}\oplus E^i_{n+1}$, we may decompose $x$
as $x=\alpha e_i^{n+1}+z$, where $z\in E^i_{n+1}$ and $\alpha\in\CC$.
Applying $d$ and using \eqref{d_e_i}
and \eqref{E_incl}, we see that $\alpha=1$. Therefore
\[
p_i^{n+1}\le \| x\|_p \le C\| y\|_q=Cq_i^n \quad (i\in I),
\]
and so $P^{[n+1]}\prec P^{[n]}$. By Lemma~\ref{lemma:Kothe_pow},
this implies that $P^{[2]}\prec P$,
and, finally, $P^{[2]}\sim P$, i.e., (\bfB) holds. The resulting contradiction shows that
$\Tor^A_n(\CC,\CC)$ is not Hausdorff, as required.
\end{proof}

The next theorem summarizes what we know about homological dimensions of
metrizable K\"othe algebras.

\begin{theorem}
\label{thm:wdgdg}
Let $A=\lambda(P)$ be a metrizable K\"othe algebra. Then
\begin{align}
\label{wdg}
\wdg A=\wdb A&=
\begin{cases}
0, & \text{$P$ satisfies {\upshape (\bfU)}}.\\
1, & \parbox[t]{78mm}{$P$ satisfies {\upshape (\bfB)} and {\upshape (\bfN)},
but does not satisfy {\upshape (\bfU)}. In this case, $\wdh_A\CC=1.$}\\
2, & \parbox[t]{78mm}{$P$ satisfies {\upshape (\bfB)},
but does not satisfy {\upshape (\bfN)}. In this case, $\wdh_A\lambda^\infty(P)=2$.}\\
\infty, & \parbox[t]{78mm}{$P$ does not satisfy {\upshape (\bfB)}. In this case,
$\wdh_A\CC=\infty$.}
\end{cases}
\\
\label{dg}
\dg A=\db A&=
\begin{cases}
0, & \text{$P$ satisfies {\upshape (\bfU)}}.\\
1, & \parbox[t]{78mm}{$P$ satisfies {\upshape (\bfB)}, {\upshape (\bfN)},
and {\upshape (\bfM)}, but does not satisfy {\upshape (\bfU)}. In this case, $\dh_A\CC=1.$}\\
2, & \parbox[t]{78mm}{$P$ satisfies {\upshape (\bfB)} and {\upshape (\bfN)},
but does not satisfy {\upshape (\bfM)}. In this case, $\dh_A\lambda(\bar P)=2$.}\\
2, & \parbox[t]{78mm}{$P$ satisfies {\upshape (\bfB)},
but does not satisfy {\upshape (\bfN)}. In this case, $\dh_A\lambda^\infty(P)=2$.}\\
\infty, & \parbox[t]{78mm}{$P$ does not satisfy {\upshape (\bfB)}. In this case,
$\dh_A\CC=\infty$.}
\end{cases}
\end{align}
\end{theorem}
\begin{proof}
As was already mentioned (see the proof of Theorem~\ref{thm:dg_bipr}),
condition (\bfB) implies that $\db A\le 2$.
Hence all the dimensions $\dg A$, $\wdg A$, and $\wdb A$ are $\le 2$.
Now \eqref{wdg} follows from
Theorem~\ref{thm:Torne0} and from \cite[Theorems 5.2 and 6.10, Proposition 6.11]{Pir_msb},
while \eqref{dg} follows from Theorems~\ref{thm:dg_bipr} and~\ref{thm:Torne0}.
\end{proof}

\begin{remark}
\label{rem:dg_Fr_Ptens}
It easily follows from Theorem~\ref{thm:wdgdg} that for a metrizable K\"othe
algebra $A=\lambda(P)$ the global dimension of $A$ does not depend on whether
we consider $A$ as a Fr\'echet algebra or as a $\Ptens$-algebra.
\end{remark}

\section{Examples}

In this section we compute homological dimensions of the K\"othe algebras
discussed in Examples~\ref{example:l^1}--\ref{example:H(D)}.
We will see, in particular, that every combination
of (\bfU), (\bfN), (\bfB), (\bfM) described in \eqref{wdg} and \eqref{dg}
is possible.

\begin{example}
The algebra $\ell^1(I)$ satisfies {\upshape (\bfB)},
but does not satisfy {\upshape (\bfN)}. Therefore for each
$d\in\{ \dg,\db,\wdg,\wdb\}$ we have $d(\ell^1(I))=2$; moreover,
we have $\dh_{\ell^1(I)} \ell^\infty(I)=\wdh_{\ell^1(I)}\ell^\infty(I)=2$.
For $\dg$, $\db$, and $\dh_{\ell^1(I)}\ell^\infty(I)$, this is an old
result by Helemskii \cite{X_ne1} (see also \cite[V.2.16]{X1});
for $\wdb$, the result is due to Selivanov \cite{Sel_bifl}.
\end{example}

\begin{example}
The algebra $\CC^I$ satisfies {\upshape (\bfU)}, and so $\dg\CC^I=\db\CC^I=0$.
If $\CC^I$ is metrizable (i.e., if $I$ is at most countable), then
$\wdg\CC^I=\wdb\CC^I=0$. These results are due to Helemskii~\cite[Theorem IV.5.27]{X1}.
Moreover, he proved [loc. cit.] that each commutative Arens-Michael algebra
$A$ with $\dg A=0$ is topologically isomorphic to $\CC^I$ for some $I$.
\end{example}

\begin{example}
The algebra $\Lambda_R(\alpha)$ (see Example~\ref{example:Lambda})
satisfies {\upshape (\bfB)} if and only if either $R=1$ or $R=\infty$
\cite[Example 3.5]{Pir_bipr}. Therefore for each $1<R<\infty$
and each $d\in\{ \dg,\db,\wdg,\wdb\}$ we have
$d(\Lambda_R(\alpha))=\dh_{\Lambda_R(\alpha)}\CC=\wdh_{\Lambda_R(\alpha)}\CC=\infty$.
In particular, $d(\cH(\DD_R))=\infty$ whenever $1<R<\infty$.
\end{example}

Before giving further examples, we would like to note that condition {\upshape (\bfM)}
is satisfied automatically for many natural K\"othe spaces.
In particular, if $I=\N$, and if $p_i\le p_{i+1}$ for each $p\in P$
and each $i\in\N$, then {\upshape (\bfM)} follows from {\upshape (\bfB)}
\cite[Corollary 7.5]{Pir_msb}.

\begin{example}
The algebra $\Lambda_\infty(\alpha)$ satisfies {\upshape (\bfB)}
and hence {\upshape (\bfM)} (see above).
Clearly, $\Lambda_\infty(\alpha)$ does not satisfy {\upshape (\bfU)}.
The Grothendieck-Pietsch criterion implies that $\Lambda_\infty(\alpha)$ satisfies
{\upshape (\bfN)} if and only if $\sup_n (\log n)/\alpha_n<\infty$
(see, e.g., \cite[29.6 and 28.16]{MV}). Therefore
for each $d\in\{ \dg,\db,\wdg,\wdb\}$ we have
\[
d(\Lambda_\infty(\alpha))=
\begin{cases}
1 & \text{if $\sup_n (\log n)/\alpha_n<\infty$}\\
2 & \text{otherwise}.
\end{cases}
\]
In particular, $d(s)=d(\cH(\CC))=1$.
\end{example}

\begin{example}
The algebra $\Lambda_1(\alpha)$ satisfies {\upshape (\bfB)}.
The Grothendieck-Pietsch criterion implies that $\Lambda_1(\alpha)$ satisfies
{\upshape (\bfN)} if and only if $\lim_n (\log n)/\alpha_n=0$
(see, e.g., \cite[29.6 and 28.16]{MV}). However, the latter condition
implies that $\Lambda_1(\alpha)$ satisfies {\upshape (\bfU)}.
Therefore
for each $d\in\{ \dg,\db,\wdg,\wdb\}$ we have
\[
d(\Lambda_1(\alpha))=
\begin{cases}
0 & \text{if $\lim_n (\log n)/\alpha_n=0$}\\
2 & \text{otherwise}.
\end{cases}
\]
In particular, $d(\cH(\DD_1))=0$.
\end{example}

\begin{example}
Let $I=\N\times\N$. For each $i,j,k\in\N$ we define
\[
p_{ij}^{(k)}=
\begin{cases}
2^{(kj)^i} (i+j)^k & (i\le k),\\
(i+j)^k & (i>k).
\end{cases}
\]
Set $p^{(k)}=(p_{ij}^{(k)})_{i,j\in\N}$, and consider the K\"othe set
$P=\{ p^{(k)}\}_{k\in\N}$. As was shown in \cite[Theorem 7.9]{Pir_msb},
$P$ satisfies {\upshape (\bfB)} and {\upshape (\bfN)},
but does not satisfy {\upshape (\bfM)}.
Note that, since $p_{ij}^{(k)}\ge 1$ for all $i,j,k$, we have $\lambda(\bar P)=\ell^1$.
Therefore
$\dg\lambda(P)=\db\lambda(P)=\dh_{\lambda(P)}\ell^1=2$,
while $\wdg\lambda(P)=\wdb\lambda(P)=\wdh_{\lambda(P)}\CC=1$.
\end{example}

\section{Remarks on quasibiprojectivity}

It is interesting to compare Theorem~\ref{thm:Torne0} with recent results of
Selivanov~\cite{Sel_inf} on quasibiprojective Banach algebras.
By definition [loc. cit.], a $\Ptens$-algebra $A$ is {\em quasibiprojective}
if $A=\overline{A^2}$ and the map
\[
\pi_A\colon A\Ptens A\to A\ptens{A} A,\quad a\otimes b\mapsto a\tens{A} b,
\]
is a retraction in $A\bimod A$. The latter condition means that there
exists an $A$-$\Ptens$-bimodule morphism
$\rho_A\colon A\ptens{A} A\to A\Ptens A$ such that $\pi_A\rho_A=\id$.
Each quasibiprojective algebra is biprojective, but the converse is false.
For example, Selivanov proved that all sequence algebras $\ell^p\; (1\le p<\infty)$ are
quasibiprojective, but are not biprojective unless $p=1$.
A similar assertion holds for the convolution algebras $L^p(G)$, where
$G$ is an infinite compact group.
The Schatten ideals $S^p(H)$ (where $H$ is an infinite-dimensional Hilbert space)
are quasibiprojective whenever $1\le p\le 2$, but are not biprojective unless $p=1$.
More examples of quasibiprojective Banach algebras can be found
in~\cite[Theorem 3.16]{Sel_inf}. As for general $\Ptens$-algebras, it is easy to show
(by using \cite[Lemma 5.1]{Pir_msb}) that all K\"othe algebras $\lambda(P)$
are quasibiprojective.

Selivanov proved that, if $A$ is a quasibiprojective, non-biprojective Banach algebra,
then $\dg A=\db A=\infty$ (see \cite[Theorem 3.14]{Sel_inf}).
In fact, an easy modification of his argument shows that $\wdg A=\wdb A=\infty$ as well
(Selivanov, private communication). It is natural to ask whether or not
Selivanov's theorem can be extended to Fr\'echet
algebras. If yes, then the last statement of our Theorem~\ref{thm:Torne0}
(except for the equality $\wdh_A\CC=\infty$) would be
an easy consequence of this general result. However, we do not know whether this
general result is true. The difficulty is that Selivanov's argument heavily
relies on some geometric properties of Banach spaces which do not hold for
nonnormable Fr\'echet spaces.


\begin{thebibliography}{99}
\bibitem{Bh_Deh}
Bhatt, S. J., and Deheri, G. M.
{\em K\"othe spaces and topological algebras with bases},
Proc. Indian Acad. Sci. Math. Sci. \textbf{100} (1990), no.~3, 259--273.
\bibitem{X_ne1}
Helemskii, A. Ya.
{\em The global dimension of a Banach function algebra is different from one}
(Russian). Funkcional. Anal. i Prilozen. \textbf{6} (1972), no. 2, 95--96;
English transl.: Functional Anal. Appl. \textbf{6} (1972), 166--168.
\bibitem{X1}
Helemskii, A. Ya. {\itshape The Homology of Banach and Topological Algebras},
Moscow University Press, 1986 (Russian); English transl.: Kluwer Academic
Publishers, Dordrecht, 1989.
\bibitem{X_HOA}
Helemskii, A. Ya.
{\itshape Homology for the Algebras of Analysis},
Handbook of algebra, Vol.~2 (ed. M.~Hazewinkel), 151--274,
Amsterdam, North-Holland, 2000.
\bibitem{MV}
Meise, R.; Vogt, D.
{\em Introduction to Functional Analysis},
Clarendon Press, Oxford, 1997.
\bibitem{Pietsch2}
Pietsch, A.
{\em Zur Theorie der topologischen Tensorprodukte},
Math. Nachr. \textbf{25} (1963), 19--31.
\bibitem{Pir_bipr}
Pirkovskii, A. Yu.
{\em Biprojective topological algebras of homological bidimension~$1$},
J. Math. Sci. (New York) \textbf{111} (2002), No.~2, 3476--3495.
\bibitem{Pir_bipr2}
Pirkovskii, A. Yu.
{\em Homological bidimension of biprojective topological algebras
and nuclearity},
Proc. of the 3rd International Conference on Topological
Algebra and Applications, Oulu University, Oulu,
Finland, July 2-6, 2001; Acta Universitatis Ouluensis A
\textbf{408} (2004), 179--196.
\bibitem{Pir_msb}
Pirkovskii, A. Yu.
{\em Weak homological dimensions and biflat K\"othe algebras} (in Russian).
Mat. Sb. \textbf{199} (2008), No.~5, 45--80; English transl.:
Sbornik: Mathematics \textbf{199} (2008), No.~5, 673--705.
\bibitem{Pir_cycflat2}
Pirkovskii, A. Yu.
{\em Flat cyclic Fr\'echet modules, amenable Fr\'echet algebras,
and approximate identities}; {\ttfamily arXiv:math.FA/0610528}.
\bibitem{Render}
Render, H., Sauer, A.
{\em Algebras of holomorphic functions with Hadamard multiplication},
Studia Math. \textbf{118} (1996), no.~1, 77--100.
\bibitem{Sel_bifl}
Selivanov, Yu. V.
{\em Weak homological bidimension and its values in
the class of biflat Banach algebras},
Extracta Math. \textbf{11} (1996), 348--365.
\bibitem{Sel_inf}
Selivanov, Yu. V.
{\em Classes of Banach algebras of global dimension infinity}.
Banach algebras and their applications, 321--333, Contemp. Math., 363,
Amer. Math. Soc., Providence, RI, 2004.
\bibitem{T1}
Taylor, J. L. {\itshape Homology and cohomology for topological algebras},
Adv. Math. \textbf{9} (1972), 137--182.
\bibitem{Vogt_seq_rep}
Vogt, D.
{\em Sequence space representations of spaces of test functions and distributions},
Functional analysis, holomorphy, and approximation theory (Rio de Janeiro, 1979),
405--443, Lecture Notes in Pure and Appl. Math., 83, Dekker, New York, 1983.
\end{thebibliography}
\end{document}